\documentclass[11pt,reqno]{amsart}
\usepackage[margin=8em, top=9em, headsep=2em ,footskip=.5in]{geometry}
\usepackage[T1]{fontenc}
\usepackage[american]{babel}
\usepackage{amsmath,amsfonts,amssymb,amsthm,bm,wasysym}
\usepackage{mathrsfs,bbm,graphicx}
\usepackage{enumitem}
\usepackage{multirow}
\usepackage{caption,float}
\usepackage{subcaption}
\usepackage{circuitikz}
\usetikzlibrary{matrix,arrows,backgrounds,calc,chains,automata,positioning,patterns}
\usetikzlibrary{decorations,decorations.pathmorphing}
\usetikzlibrary{shapes.geometric,calc}
\usepackage[colorlinks,linkcolor=cyan!50!black,citecolor=cyan!30!black,pagebackref,hypertexnames=false, breaklinks]{hyperref}

\allowdisplaybreaks
\newtheorem{theorem}{Theorem}
\newtheorem{proposition}[theorem]{Proposition}

\newtheorem{lemma}[theorem]{Lemma}
\theoremstyle{definition}

\newtheorem*{def*}{Definition}

\newtheorem{remark}{Remark}

\newtheorem*{thm*}{Theorem}
\newtheorem*{lem*}{Lemma}
\newtheorem*{claim*}{Claim}
\newtheorem*{prop*}{Proposition}
\newtheorem*{rem*}{Remark}

\numberwithin{equation}{section}

\title[Dyadic fractional Sobolev]{Fractional Sobolev embeddings and algebra property: A dyadic view}
\author{Patricia Alonso Ruiz}
\address{Institute of Mathematics, Friedrich Schiller University Jena, Germany}
\email{patricia.alonso.ruiz@uni-jena.de}
\author{Valentia Fragkiadaki}
\address{Department of Mathematics, Texas A{\&}M University, US}
\email{valeria96@tamu.edu}
\thanks{Research partly supported by NSF grant DMS~2140664.}

\makeindex
\begin{document}
\begin{abstract}
    This paper revisits classical fractional Sobolev embedding theorems and the algebra property of the fractional Sobolev space $H^s(\mathbb{R})$ by means of Haar functions and dyadic decompositions. The aim is to provide an alternative, hands-on approach without Fourier transform that may be transferred to settings where the latter is not available. Explicit counterexamples are constructed to show the failure of the algebra property in the low-regularity regime.
\end{abstract}
\maketitle

{\small
}
\tableofcontents


\section{Introduction}

\medskip

Calculus of variations, partial differential equations, harmonic analysis, stochastic processes... a wide variety of areas put their hands on fractional Sobolev spaces, for instance, to study questions involving non-local operators such as the fractional Laplacian. There are a number of ways to characterize these spaces in the Euclidean case. One can use the Fourier transform and Littlewood–Paley decomposition, interpolation theory, singular integrals, and also semigroups and subordination. There is a big fractional world out there!

\bigskip

In this paper, we are going to approach the fractional Sobolev space $H^s(\mathbb{R})$ with $0<s<1$ from a dyadic point of view and revisit two important aspects of them: Sobolev embeddings and the algebra property. Why these? And why the dyadic setting? 

\bigskip

Sobolev embeddings are a basic tool from functional analysis that have extensive applications, in particular in the theory of PDEs. Their scope would be difficult to capture here, and we refer to textbooks like~\cite{Eva10} or~\cite{Leo23} for further details. In the form presented in the present paper, the embeddings state that
\begin{equation}\label{E:classical_embedding}
    H^{s}_{\mathcal{D}}(\mathbb{R})\subseteq
    \begin{cases}
        L^{\frac{2}{1-2s}}(\mathbb{R})&\text{if }0<s<\frac{1}{2},\\
        {\rm BMO_{\mathcal{D}}}(\mathbb{R})&\text{if }s=\frac{1}{2},\\
        L^\infty(\mathbb{R})&\text{if }s>\frac{1}{2}.
    \end{cases}
\end{equation}
Here, ${\rm BMO_{\mathcal{D}}}(\mathbb{R})$ refers to the dyadic space of \emph{bounded mean oscillation} functions in $\mathbb{R}$ for a dyadic grid $\mathcal{D}$, c.f. Section~\ref{SS:BMO}.

\bigskip

The more interesting question for us is the algebra property of the space $H^s(\mathbb{R})$, which concerns the question, whether the space is closed under pointwise multiplication. Our interest in this property, proved first by Strichartz in~\cite{Str67} and strengthened later by Kato and Ponce in~\cite{KP88}, stems from its importance when studying well-posedness in $H^s(\mathbb{R})$ of nonlinear PDEs, see e.g.~\cite{Caz03} and references therein. In its original form, the result states that 

\begin{equation}\label{E:classical_algebra}
    H^s(\mathbb{R})\text{ is an algebra if and only if }s>1/2.
\end{equation}

\bigskip

And why the dyadic setting? Because we wanted to investigate an alternative approach that would bypass the Fourier transform and avoid abstract (yet powerful) arguments, as in the case of interpolation or semigroup theory. We were greatly motivated to write this paper after reading a series of works by Aimar and coauthors~\cite{ABG13,AA15,ACGN23}, where we learned about the fractional Sobolev space $H^s(\mathbb{R})$ in the dyadic setting. After we finished the paper, we found out about work by Garrig\'os, Seeger and Ullrich dealing with dyadic versions of Triebel-Lizorkin spaces and their embeddings~\cite{GSU17,GSU18,GSU23}. Since $H^s(\mathbb{R})$ is a special case of such a space, it will be interesting to analyze and compare both approaches in the future. Looking ahead, we hope that the techniques involved in this paper can offer a hands-on approach to these and related questions in other settings where no Fourier transform is available.

\bigskip

The paper is organized as follows: Section~\ref{S:defs_back} provides the main definitions and background information regarding the standard dyadic system in $\mathbb{R}$ and the corresponding definition of the dyadic fractional Sobolev space $H^s_{\mathcal{D}}(\mathbb{R})$. The section also includes some useful inequalities of independent interest. Section~\ref{S:embeddings} presents the dyadic proofs of the fractional Sobolev embeddings~\eqref{E:classical_embedding}, treating each regularity range separately. With a similar structure, Section~\ref{S:algebra_prop} proves the classical algebra property~\eqref{E:classical_algebra}, constructing explicit counterexamples in the cases where the property fails. To the best of our knowledge, these are genuinely new.

\section{Definitions and background}\label{S:defs_back}
In this section we set up notation, definitions and some basic results in the dyadic setting that will be used throughout the paper. 
\subsection{The dyadic setting}
A dyadic grid $\mathcal{D}$ on $\mathbb{R}$ is a collection of intervals that satisfies:
	\begin{enumerate}[label={\rm (D\arabic*)}]
	\item Every $I\in\mathcal{D}$ has length $\ell(I)=2^k$ for some $k\in \mathbb{Z}$,
	\item For any fixed $k_0\in \mathbb{Z}$, the sub-collection $\{I\in\mathcal{D}: \ell(I)=2^{k_0} \}$ forms a partition of $\mathbb{R}$,
	\item Any two dyadic intervals are either disjoint, or one contains the other, i.e. $I\cap J \in \{ \varnothing , I,J \}$ for all $I,J \in \mathcal{D}$.
	\end{enumerate}
Given a dyadic grid $\mathcal{D}$, an interval $I\in\mathcal{D}$ and $k\geq 0$, the dyadic grid of level $k$ adapted to $I$ is defined as
\begin{equation}\label{E:def_grid_level_k}
    \mathcal{D}_k(I):=\Big\{J\in \mathcal{D}\colon J\subset I\text{ and }|J|=2^{-k}|I|\Big\},
\end{equation}
where $|\cdot|$ indicates the Lebesgue measure on $\mathbb{R}$. We usually write $\mathcal{D}(I):=\mathcal{D}_0(I)$.

\medskip

From now on, we consider a dyadic grid $\mathcal{D}$ in $\mathbb{R}$ and its associated Haar basis of $L^2(\mathbb{R},dx)$, which consists of the functions $\{h_I\}_{I\in\mathcal{D}}$ given by
\begin{equation}\label{E:def_Haar_fct}
    h_I(x):= \frac{1}{\sqrt{|I|}}(\mathbf{1}_{I_+} - \mathbf{1}_{I_-}),
\end{equation}
where $I_+$ denotes the right half and $I_-$ the left half  of $I$.
Note that, if $J \in \mathcal{D}$ and $J\subsetneq I$, then $h_I$ is constant on $J$. We will often denote that constant by $h_I(J)$, specifically
\begin{equation}\label{E:hI_in_J}
    h_I(J):=h_I(x)|_J=
    \begin{cases}
        |I|^{-1/2} & \text{if  } J \subset I_+\\
        -|I|^{-1/2} & \text{if  } J \subset I_-.
    \end{cases}
\end{equation}

The role of Schwarz functions is played by the linear span of the Haar basis $\{h_I\}_{I\in\mathcal{D}}$, that is
\begin{equation}\label{E:def_Haar_span}
    \mathscr{S}_{\mathcal{D}}(\mathbb{R}):={\rm span}\{h_I\colon I\in \mathcal{D}\}.
\end{equation}
The latter is dense in $L^2(\mathbb{R},dx)$, hence for simplicity the results presented in the paper will be stated and proved for functions in $\mathscr{S}_{\mathcal{D}}(\mathbb{R})$. In addition, the $L^2$-inner product of two functions $f,g\in \mathscr{S}_{\mathcal{D}}(\mathbb{R})$ will be denoted by
\begin{equation}\label{E:def_inner_prod_average}
    (f,g):=\int_{\mathbb{R}}fg\,dx
\end{equation}
and the average over an interval $I\in\mathcal{D}$ by
\begin{equation}\label{E:def_average}
    \langle f\rangle_I:= \frac{1}{|I|} \int_I f \,dx.
\end{equation}
The proofs developed in the subsequent section rely on the properties of the Haar expansion 
\begin{equation*}
    f=\sum_{I\in\mathcal{D}} (f, h_I) h_I 
\end{equation*}
and the arising expression for the average
\begin{equation}\label{E:avg_char}
    \langle f\rangle_I = \sum_{J \supsetneq I} (f, h_J) h_J(I).
\end{equation}
\subsection{Useful equalities}
This paragraph collects several observations that will be used in the proofs and may be of independent interest to the reader.

\begin{proposition}\label{P:weighted_indicator}
Let $s>0$ and $I\in\mathcal{D}$. Then,
\begin{equation}\label{E:weighted_indicator}
    \sum_{J\subseteq I}|J|^s\mathbf{1}_J(x)=\frac{1}{1-2^{-s}}|I|^s\mathbf{1}_I(x)
\end{equation}
for any $x\in\mathbb{R}$.
\end{proposition}
\begin{proof}
    If $x\notin I$, then clearly both sides of~\eqref{E:weighted_indicator} are zero. Let us thus assume that $x\in I$. Then, rewriting the left-hand side with the dyadic grid from~\eqref{E:def_grid_level_k} one gets
    \begin{align*}
     \sum_{J\subseteq I}|J|^s\mathbf{1}_J(x)&=\sum_{k=0}^\infty \sum_{J\in\mathcal{D}_k(I)}|J|^s\mathbf{1}_J(x)\\
     &=\sum_{k=0}^\infty\sum_{J\in \mathcal{D}_k(I)}2^{-ks}|I|^s\mathbf{1}_J(x)\\
     &=|I|^s\sum_{k=0}^\infty2^{-ks}\bigg(\sum_{J\in\mathcal{D}_k(I)}\mathbf{1}_J(x)\bigg)
     =|I|^s\frac{1}{1-2^{-s}},
    \end{align*}
    where in the last equality we use the fact that the intervals in $\mathcal{D}_k(I)$ are disjoint.
\end{proof}

\begin{proposition}\label{P:weighted_hI}
 Let $s>0$ and $I\in\mathcal{D}$. Then,
 \begin{equation}\label{E:weighted_hI}
     \sum_{J\subsetneq I}|J|^sh_I(J)\mathbf{1}_J(x)=\frac{1}{2^s-1}|I|^sh_I(x)
 \end{equation}
 for every $x\in \mathbb{R}$.
\end{proposition}
\begin{proof}
    By virtue of Proposition~\ref{P:weighted_indicator},
    \begin{align*}
      \sum_{J\subsetneq I}|J|^sh_I(J)\mathbf{1}_J(x)&
      =\frac{1}{\sqrt{|I|}}\bigg(\sum_{J\subseteq I_{+}}|J|^s\mathbf{1}_J(x)-\sum_{J\subseteq I_{-}}|J|^s\mathbf{1}_J(x)\bigg)\\
      &=\frac{1}{1-2^{-s}}\frac{1}{\sqrt{|I|}}\Big(|I_{+}|^s\mathbf{1}_{I_+}(x)-|I_{-}|^s\mathbf{1}_{I_{-}}(x)\Big).
    \end{align*}
    Since $|I_{+}|=|I|/2$, it follows from the above and the definition of $h_I$ that
    \[
    \sum_{J\subsetneq I}|J|^sh_I(J)\mathbf{1}_J(x)=\frac{1}{2^s-1}|I|^s\frac{1}{\sqrt{|I|}}\Big(\mathbf{1}_{I_+}(x)-\mathbf{1}_{I_{-}}(x)\Big)=\frac{1}{2^s-1}|I|^sh_I(x)
    \]
    as we wanted to prove.
\end{proof}

\subsection{Fractional Sobolev spaces}
The definition of fractional Sobolev spaces with which we work has been considered in the dyadic setting to study nonlocal Schr\"odinger operators; see, e.g.~\cite{ABG13,AABG16,ACGN23} and references therein. Its definition is inspired by the classical (Euclidean) connection in harmonic analysis between the fractional Laplace operator $(-\Delta)^{s/2}$ and the singular operator
\begin{equation}\label{E:def_Is}
I_sf(x):=\int_{\mathbb{R}}\frac{1}{|x-y|^{1-s}}f(y)dy,
\end{equation}
when $0<s<1$, namely
\begin{equation}\label{E:Muscalu-Schlag}
    f=C_sI_s(-\Delta)^{s/2}f
\end{equation}
for some (explicit) constant $C_s>0$ and any Schwarz function $f$, see e.g.~\cite[p.177]{MS13}. The dyadic analogue operator to~\eqref{E:def_Is} defined as
\begin{equation}
T_{\mathcal{D},s}f:=  \sum_{I \in \mathcal{D}} |I|^s \langle f \rangle_I \mathbf{1}_I,
\qquad f\in\mathscr{S}(\mathbb{R}),
\end{equation}
has since long been studied in the literature, see e.g.~\cite{MW74,Saw88,HRS16} and references therein. The dyadic analogue to the fractional derivative $(-\Delta)^{s/2}$ that we consider is defined as
\begin{equation}\label{E:def_Ds}
    D^s_{\mathcal{D}}f:=\sum_{I\in\mathcal{D}}|I|^{-s}(f,h_I)h_I(x).
\end{equation}
Along the lines of the classical definition of the fractional Sobolev space $H^s(\mathbb{R})$ via Fourier transform, one can thus define the seminorm
\begin{equation}\label{E:def_Hs_seminorm}
    \|f\|_{\dot{H}^s_{\mathcal{D}}(\mathbb{R})}:=\|D^s_{\mathcal{D}}f\|_{L^2(\mathbb{R})}= \bigg( \sum_{I \in \mathcal{D}} |I|^{-2s}|(f,h_I)|^2 \bigg)^{1/2} 
\end{equation}
and the dyadic fractional Sobolev space
\begin{equation}
    H_{\mathcal{D}}^s(\mathbb{R}):=\{f\in L^2(\mathbb{R})~\colon~D^s_{\mathcal{D}}f\in L^2(\mathbb{R})\}
\end{equation}
equipped with the norm 
\begin{equation}\label{E:def_Hs_norm}
    \|f\|_{H^s_{\mathcal{D}}(\mathbb{R})}:=\big(\|f\|_{\dot{H}^s_{\mathcal{D}}(\mathbb{R})}^2+\|f\|_{L^2(\mathbb{R})}^2\big)^{1/2},
\end{equation}
see~\cite[Theorem 4.1]{AA15}. We now show the connection between the dyadic fractional derivative $D^s$ and the operator $T_{\mathcal{D},s}$ in~\eqref{E:Muscalu-Schlag}.

\begin{lemma}\label{L:dyadic_MS}
    Let $0<s<1$. For any $f\in \mathscr{S}_{\mathcal{D}}^s(\mathbb{R})$,
    \[
    f=(2^s-1)T_{\mathcal{D},s}D_{\mathcal{D}}^sf.
    \]
\end{lemma}
\begin{proof}
Consider first $g\in\mathscr{S}_{\mathcal{D}}^s(\mathbb{R})$. Using the Haar expansion of the average from~\eqref{E:avg_char}, interchanging the order of summation and applying Proposition~\ref{P:weighted_hI},
    \begin{align*}
        (2^s-1)T_{\mathcal{D},s}g&=(2^s-1)\sum_{I \in \mathcal{D}} |I|^s \langle g \rangle_I \mathbf{1}_I\\
        &=(2^s-1)\sum_{I \in \mathcal{D}} |I|^s\mathbf{1}_I(x)\sum_{J\supsetneq I}(g,h_J)h_{J}(I)\\
        &=\sum_{J\in\mathcal{D}}(g,h_J)(2^s-1)\sum_{I\subsetneq J}|I|^sh_J(I)\mathbf{1}_I(x)\\
        &=\sum_{J\in\mathcal{D}}|J|^s(g,h_J)h_J.
    \end{align*}
    Plugging $g=D^s_{\mathcal{D}}f$ the result follows.
\end{proof}

We close this section with an observation that will be useful in subsequent proofs. Namely, that only small size intervals matter when computing the $H^s$-norm~\eqref{E:def_Hs_seminorm} for small $s$.
\begin{lemma}\label{L:norm_equivalence}
    Let $0<s<1$. For any $f\in \mathscr{S}_{\mathcal{D}}(\mathbb{R})$ 
    \begin{equation}\label{E:seminorm_equivalence}
    \frac{1}{2}\|f\|_{H^s_{\mathcal{D}}(\mathbb{R})}^2\leq \|f\|_{L^2(\mathbb{R})}^2 + \sum_{\substack{I\in \mathcal{D} \\ |I|<1}} |I|^{-2s}|(f,h_I)|^2\leq \|f\|_{H^s_{\mathcal{D}}(\mathbb{R})}^2.
    \end{equation}
\end{lemma}
\begin{proof}
    First, notice that 
    \[
    \|f\|_{L^2(\mathbb{R})}^2 + \sum_{\substack{I\in \mathcal{D} \\ |I|<1}} |I|^{-2s}|(f,h_I)|^2 \leq \|f\|_{L^2(\mathbb{R})}^2 +\sum_{I\in \mathcal{D}} |I|^{-2s}|(f,h_I)|^2=\|f\|_{H^s_{\mathcal{D}}(\mathbb{R})}^2
    \]
    since all terms inside the last summation are positive.
    Moreover, 
    \begin{align*}
        \|f\|_{H^s_{\mathcal{D}}(\mathbb{R})}^2 & = \|f\|_{L^2(\mathbb{R})}^2 + \sum_{\substack{I\in \mathcal{D} \\ |I|<1}} |I|^{-2s}|(f,h_I)|^2 + \sum_{\substack{I\in \mathcal{D} \\ |I|\geq 1}} |I|^{-2s}|(f,h_I)|^2 \\
        & \leq \|f\|_{L^2(\mathbb{R})}^2 + \sum_{\substack{I\in \mathcal{D} \\ |I|<1}} |I|^{-2s}|(f,h_I)|^2 + \sum_{\substack{I\in \mathcal{D} \\ |I|\geq 1}} |(f,h_I)|^2 \\
        & \leq \|f\|_{L^2(\mathbb{R})}^2 + \sum_{\substack{I\in \mathcal{D} \\ |I|<1}} |I|^{-2s}|(f,h_I)|^2  + \sum_{I\in \mathcal{D}} |(f,h_I)|^2 \\
        & \leq 2\|f\|_{L^2(\mathbb{R})}^2 + 2 \sum_{\substack{I\in \mathcal{D} \\ |I|<1}} |I|^{-2s}|(f,h_I)|^2
    \end{align*}
    which proves the result.
\end{proof}

\begin{remark}\label{R:notation}
    For ease of notation, we will drop the subscript $\mathcal{D}$ from norms and operators as long as no confusion may occur. 
\end{remark}

\section{Fractional embeddings}\label{S:embeddings}
In this section we present the dyadic analogue to the classical fractional embeddings for $p=2$, see e.g.~\cite[Chapter 2]{Leo23}. The proofs also rely on purely dyadic methods, and each regularity range is treated separately for the reader's convenience. Throughout the section, the expression $A\apprle_s B$ means that $A$ is bounded by a constant that depends on $s$ times $B$, and $A\simeq B$ means that $A$ is bounded above and below by constant times $B$.
\subsection{\texorpdfstring{Gagliardo-Nirenberg-Sobolev: $0<s<1/2$}{[small]}}
We start by proving the analogue to the Sobolev-Gagliardo-Nirenberg embedding, see e.g.~\cite[Theorem 2.2]{Leo23}. To this end, Lemma~\ref{L:dyadic_MS} will allow us to apply known estimates for the operator $T_s$ to one of the classical proofs of the Euclidean fractional Sobolev embedding in the range $0<s<1/2$.
\begin{proposition}\label{P:Embedding_below}
    Let $0<s<\frac{1}{2}$. Then, $\displaystyle H^s(\mathbb{R})\subseteq L^{\frac{2}{1-2s}}(\mathbb{R})$. In particular,
    \begin{equation}\label{E:embedding_low}
    \|f\|_{L^{q}(\mathbb{R},dx)}\apprle_s \|f\|_{\dot{H}^s(\mathbb{R})}
    \end{equation}
    for any $f\in \mathscr{S}(\mathbb{R})$ with $q:=\frac{2}{1-2s}$.
\end{proposition}
\begin{proof}
    Note first that $\frac{1}{q}=\frac{1}{2}-\frac{2s}{2}$. We prove~\eqref{E:embedding_low}.
    By virtue of Lemma~\ref{L:dyadic_MS} and the boundedness of the operator $T_s$, c.f.~\cite[Theorem 4]{MW74}, 
    \[
    \|f\|_{L^q}=(2^s-1)\|T_s(D^{s}f)\|_{L^q}\apprle_q \|D^{s}f\|_{L^2(\mathbb{R})}.
    \]
    In particular $f\in L^{\frac{2}{1-2s}}(\mathbb{R})$.
\end{proof}
\subsection{\texorpdfstring{Sobolev-Poincar\'e: $s=1/2$}{[critical]}}\label{SS:BMO}
The fractional Sobolev embedding from Proposition~\ref{P:Embedding_below} fails at the borderline case $s=1/2$. In this section, we see that the correct space is that of dyadic bounded mean oscillation, ${\rm BMO_{\mathcal{D}}}(\mathbb{R})$. The proof will use the characterization of the BMO-norm in the dyadic setting, see for example \cite[Section 1.5]{Per01}, given by
\begin{equation}\label{E:BMO_char}
    \|f\|_{\rm BMO_{\mathcal{D}}} \simeq \sup_I \bigg( \frac{1}{|I|}\sum_{ \substack{J\in\mathcal{D} \\ J\subseteq I}} (f, h_J)^2 \bigg)^{1/2}.
\end{equation}
\begin{proposition}\label{P:embedding_BMO}
    It holds that $H^{1/2}(\mathbb{R})\subseteq {\rm BMO}(\mathbb{R})$, where both spaces are considered in the dyadic setting.
\end{proposition}
\begin{proof}
    In view of~\eqref{E:BMO_char}, it suffices to prove that for any $f\in\mathscr{S}$ and $I\in\mathcal{D}_0$
    \[
    \frac{1}{|I|}\sum_{J\subseteq I}|(f,h_J)|^2\leq \sum_{J\in \mathcal{D}}|J|^{-1}|(f,h_J)|^2.
    \]
    However, the later is clear because
    \[
    \frac{1}{|I|}\sum_{J\subseteq I}|(f,h_J)|^2=\sum_{J\subseteq I}\frac{|J|}{|I|}|J|^{-1}||(f,h_J)|^2\leq \sum_{J\subseteq I}|J|^{-1}||(f,h_J)|^2\leq \sum_{J\in \mathcal{D}}|J|^{-1}|(f,h_J)|^2.
    \]
\end{proof}
\subsection{\texorpdfstring{Morrey: $s>1/2$}{[large]}}
We finish this section with a dyadic version of the classical Morrey-type embedding in the higher-regularity range. The proof is based on the following property of the dyadic averages $\langle f\rangle_I$. An alternative proof of the embedding via integration can be found in~\cite[Lemma 3.2]{AA15}.

\begin{lemma}\label{L:telescopic_averages}
Let $f\in L^2(\mathbb{R})$, $I\in\mathcal{D}$ and $k\geq 0$. Then,
    \begin{equation}\label{E:telescopic_averages}
    \langle f\rangle_{I}-\langle f\rangle_{I_{(k)}}=\sum_{I\subsetneq J\subseteq I_{(k)}}(f,h_J)h_J(I),
    \end{equation}
    where $I_{(k)}\in\mathcal{D}_k(I)$ denotes the $k$th dyadic ancestor of $I$.
\end{lemma}
\begin{proof}
    Using the expression~\eqref{E:avg_char},
    \begin{equation*}
    \langle f\rangle_I=\sum_{J\supsetneq I}(f,h_J)h_J(I)
    =\sum_{I\subsetneq J\subseteq I_{(k)}}(f,h_J)h_J(I)+
    \sum_{J\supsetneq I_{(k)}}(f,h_J)h_J(I).
    \end{equation*}
    Note that $I\subset I_{(k)}$ implies $h_J(I)=h_J(I_{(k)})$ for all $k\geq 0$ and therefore, in view of~\eqref{E:avg_char},
    \[
    \sum_{J\supsetneq I_{(k)}}(f,h_J)h_J(I)=\sum_{J\supsetneq I_{(k)}}(f,h_J)h_J(I_{(k)})=\langle f\rangle_{I_{(k)}}
    \]
    whence the result follows.
\end{proof}
\begin{proposition}\label{P:Morrey}
    Let $s>1/2$. Then it holds that $H^s(\mathbb{R})\subseteq L^\infty(\mathbb{R})$. In particular,
    \begin{equation}\label{E:Morrey_01}
        \|f\|_{L^\infty(\mathbb{R})}\apprle_s \|f\|_{H^s(\mathbb{R})}
    \end{equation}
    for any $f\in \mathscr{S}(\mathbb{R})$.
\end{proposition}
\begin{proof}
    Let $f\in \mathscr{S}(\mathbb{R})$ and $x\in\mathbb{R}$ be a Lebesgue point, i.e.
    \begin{equation}\label{E:Lebesgue_point}
        f(x)=\lim_{k\to\infty}\langle f\rangle_{I_x^{(k)}},
    \end{equation}
    where $I_x^{(k)}\in \mathcal{D}$ denotes the dyadic interval of length $2^{-k}$ that contains $x$. Noticing that $|I_x^{(0)}|=1$, the triangle inequality and Cauchy-Schwarz yield
    \begin{align*}
        |f(x)|&\leq |f(x)-\langle f\rangle_{I_x^{(0)}}|+|\langle f\rangle_{I_x^{(0)}}|\\
        &\leq |f(x)-\langle f\rangle_{I_x^{(0)}}|+\bigg(\int_{I_x^{(0)}}|f|^2dx\bigg)^{1/2}|I_x^{(0)}|^{1/2}\\
        &\leq |f(x)-\langle f\rangle_{I_x^{(0)}}|+\|f\|_{L^2}.
    \end{align*}
    Thus,~\eqref{E:Morrey_01} will follow once we show
    \begin{equation}\label{E:Morrey_02}
        |f(x)-\langle f\rangle_{I_x^{(0)}}|\apprle_s \bigg(\sum_{I\in\mathcal{D}}|I|^{-2s}|(f,h_{I})|^2\bigg)^{1/2}.
    \end{equation}
    In view of~\eqref{E:Lebesgue_point}, writing the difference on the left-hand side of~\eqref{E:Morrey_02} as a telescopic sum and applying Lemma~\ref{L:telescopic_averages} we get
    \begin{align*}
    |f(x)-\langle f\rangle_{I_x^{(0)}}|&\leq \sum_{k=0}^\infty |\langle f\rangle_{I_x^{(k)}}-\langle f\rangle_{I_x^{(k+1)}}|\\
    &\leq \sum_{k=0}^\infty\; \sum_{I_x^{(k+1)}\subsetneq J\subseteq I_x^{(k)}}|(f,h_J)|\,|h_J(I_x^{(k+1)})|\\
    &=\sum_{k=0}^\infty |(f,h_{I_x^{(k)}})|\,|h_{I_x^{(k)}}(I_x^{(k+1)})|\\
    &=\sum_{k=0}^\infty |(f,h_{I_x^{(k)}})|\,|I_x^{(k+1)}|^{-1/2},
    \end{align*}
    where in the last equality we used~\eqref{E:hI_in_J}. Rewriting and applying Cauchy-Schwarz to the latter yields
    \begin{align*}
    |f(x)-\langle f\rangle_{I_x^{(0)}}|&\leq\sum_{k=0}^\infty |I_x^{(k)}|^{-s}|(f,h_{I_x^{(k)}})|\,|I_x^{(k)}|^{-\frac{1}{2}+s}\\
    &\leq \bigg(\sum_{I\in\mathcal{D}}|I|^{-2s}|(f,h_{I})|^2\bigg)^{1/2}\bigg(\sum_{k=0}^\infty |I_x^{(k)}|^{-1+2s}\bigg)^{1/2}\\
    &=\|f\|_{\dot{H}^s}(1-2^{1-2s})^{-1/2}
    \end{align*}
    as we wanted to prove.    
\end{proof}
\section{Algebra property}\label{S:algebra_prop}
We now turn to the second main question of this paper, establishing the validity or failure of the algebra property of the space $H^s(\mathbb{R})$. The question reduces to finding out whether $f\in H^s$ implies $f^2\in H^s(\mathbb{R})$. In fact, if the latter holds and $f,g\in H^s(\mathbb{R})$, then $f+g\in H^s(\mathbb{R})$ and hence $(f+g)^2\in H^s(\mathbb{R})$ from which one concludes $fg\in H^s(\mathbb{R})$.

\medskip

As we shall see, the algebra property is closely related to the fractional Sobolev embedding: it fails in the lower-regularity range $0<s\leq 1/2$, for which we provide a counterexample, and it holds in the higher-regularity range $1/2<s<1$.

\medskip

The function $f^2$ thus plays a special role in this section, and we record the following expression of its Haar expansion.

\begin{lemma}\label{L:Haar_exp_square}
    For any $f\in\mathscr{S}(\mathbb{R})$,
    \begin{equation}\label{E:Haar_exp_square}
        f^2 = \sum_{\substack{I,J\in \mathcal{D} \\ I \subsetneqq J}} (f,h_I)^2 h_J(I)h_J + 2\sum_{\substack{I,J\in \mathcal{D} \\ I \subsetneqq J}} (f,h_I) (f,h_J) h_J(I) h_I.
    \end{equation}
\end{lemma}

\begin{proof}
    To see this, write $f=\sum_{I\in\mathcal{D}} (f,h_I)h_I$ and expand the square. Recalling that $h_I^2=|I|^{-1}\mathbf{1}_I$ and using the Haar expansion of the characteristic function, i.e. 
    \[
    \mathbf{1}_I = \sum_{J\in \mathcal{D}}(\mathbf{1}_I, h_J)h_J = \sum_{J \supsetneq I} h_J(I)|I|h_J,
    \]
    we get the result.
\end{proof}
As a consequence of the latter and in view of~\eqref{E:avg_char}, the Haar coefficient of $f^2$ associated with a dyadic interval $K\in\mathcal{D}$ reads
\begin{equation}\label{E:Haar_coef_square}
    \begin{aligned}
    (f^2,h_K)&=\sum_{\substack{I\in\mathcal{D}\\ I\subsetneq K}}(f,h_I)^2h_K(I)+2\sum_{\substack{I\in\mathcal{D}\\ I\supsetneq K}}(f,h_K)(f,h_I)h_I(K)\\
    &=\sum_{I\subsetneq K}(f,h_I)^2h_K(I)+2(f,h_K)\sum_{I\supsetneq K}(f,h_I)h_I(K)\\
    &=\sum_{I\subsetneq K}(f,h_I)^2h_K(I)+2(f,h_K)\langle f\rangle_K.
    \end{aligned}
\end{equation}
\subsection{Low regularity: \texorpdfstring{$0<s<1/2$}{[small]}}
To establish the failure of the algebra property in this range we construct a counterexample, for which we use the following notation: Let $I^{(0)}=(-1,0)$ and for any $k\geq 1$, let $I^{(k)}$ denote the right dyadic child of $I^{(k-1)}$, see Figure~\ref{F:left_dyadic}. In particular, $I^{(k)}\subset I^{(m)}_{+}$ for any $k>m\geq 0$.

\begin{figure}[H]
    \centering
\begin{tikzpicture}
\coordinate (S) at ($(0:0)$);
\coordinate (E) at ($(0:12)$);
\draw[densely dotted] (S) -- (E);
\coordinate[label=below: {\footnotesize $-1$}] (So) at ($(0:2.5)$);
\coordinate[label=above: {\footnotesize $I^{(0)}$}] (Eo) at ($(0:6)$);
\coordinate[label=below: {\footnotesize $0$}] (Eo) at ($(0:9.5)$);
\draw[(-)] (So) -- (Eo);
\coordinate (S1) at ($(0:0)+(270:1)$);
\coordinate (E1) at ($(0:12)+(270:1)$);
\draw[densely dotted] (S1) -- (E1);
\coordinate[label=below: {\footnotesize $0$}] (So1) at ($(0:9.5)+(270:1)$);
\coordinate[label=above: {\footnotesize $I^{(1)}$}] (Eo1) at ($(0:8)+(270:1)$);
\coordinate[label=below: {\footnotesize $-1/2$}] (Eo1) at ($(0:6)+(270:1)$);
\draw[(-)] (So1) -- (Eo1);
\coordinate (S2) at ($(0:0)+(270:2)$);
\coordinate (E2) at ($(0:12)+(270:2)$);
\draw[densely dotted] (S2) -- (E2);
\coordinate[label=below: {\footnotesize $0$}] (So2) at ($(0:9.5)+(270:2)$);
\coordinate[label=above: {\footnotesize $I^{(2)}$}] (Eo2) at ($(0:8.75)+(270:2)$);
\coordinate[label=below: {\footnotesize $-1/4$}] (Eo2) at ($(0:7.75)+(270:2)$);
\draw[(-)] (So2) -- (Eo2);
\coordinate (D1) at ($(0:6)+(270:2.5)$);
\coordinate (D2) at ($(0:6)+(270:3.25)$);
\draw[dotted, thick] (D1) -- (D2);
\end{tikzpicture}
    \caption{Dyadic intervals involved in the counterexample~\eqref{E:counterex_low_reg}.}
    \label{F:left_dyadic}
\end{figure}
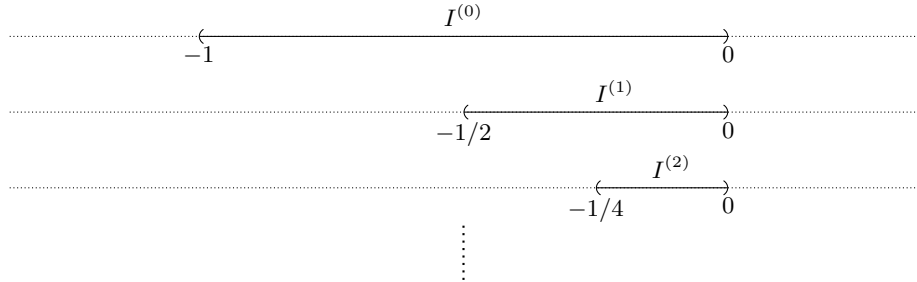
\begin{proposition}\label{P:couterex_low_reg}
    The space $H^s(\mathbb{R})$ is not an algebra when $0<s<1/2$. In particular, the function $f\in\mathscr{S}(\mathbb{R})$ given by
    \begin{equation}\label{E:counterex_low_reg}
        f:=\sum_{k=0}^\infty |I^{(k)}|^\alpha h_{I^{(k)}},
    \end{equation}
    where $s<\alpha<\frac{s}{2}+\frac{1}{4}$, belongs to $H^s(\mathbb{R})$ while $f^2\notin H^s(\mathbb{R})$.
\end{proposition}

\begin{remark}
Note that $s<\frac{s}{2}+\frac{1}{4}$ in the range $0<s<1/2$, so that $\alpha$ can be chosen in the desired range.
\end{remark}

\begin{proof}
    We first check that the function $f$ in~\eqref{E:counterex_low_reg} belongs to $H^s(\mathbb{R})$. By definition, the Haar coefficients of $f$ are
    \begin{equation}\label{E:counterex_low_reg_Haar_coeffs}
        (f,h_I)=\begin{cases}
            |I^{(k)}|^\alpha=2^{-k\alpha}&\text{if }I=I^{(k)}\text{ for some }k \geq 0,\\
            0&\text{otherwise}.
        \end{cases}
    \end{equation}
    Thus,
    \begin{align*}
    \| D^sf \|_{L^2}^2 
    & =  \sum_{I\in \mathcal{D}} |I|^{-2s} |(f,h_I)|^2 \\
    & = \sum_{k\geq 0} |I^{(k)}|^{-2s} |(f,h_{I^{(k)}})|^2 \\
    & = \sum_{k\geq 0} 2^{2ks} \, 2^{-2k\alpha}= \sum_{k\geq 0} 2^{2k(s-\alpha)}
    \end{align*}
    and the latter converges because $\alpha>s$.
    
    With a slight abuse of notation, we denote by $I^{(m)}$ with $m<0$ the dyadic ancestor of $I^{(m+1)}$. In this way, $I^{(k)}\subset I^{(m)}_{+}$ also holds for any $k\geq 0$ and $m<0$, and therefore $h_{I^{(m)}}(I^{(k)})=|I^{(m)}|^{-1/2}=2^{m/2}$ for such indices. Substituting in Lemma~\ref{L:Haar_exp_square} the coefficients from~\eqref{E:counterex_low_reg_Haar_coeffs} yields
    \begin{align*}
        f^2&=\sum_{m\in\mathbb{Z}}\sum_{k>\max\{0,m\}}(f,h_{I^{(k)}})^2h_{I^{(m)}}(I^{(k)})h_{I^{(m)}}
        +2\sum_{m\geq 0}\sum_{k>m}(f,h_{I^{(k)}})(f,h_{I^{(m)}})h_{I^{(m)}}(I^{(k)})h_{I^{(k)}}\\
        &=\sum_{m\in\mathbb{Z}}\sum_{k>\max\{0,m\}}(f,h_{I^{(k)}})^2h_{I^{(m)}}(I^{(k)})h_{I^{(m)}}
        +2\sum_{k> 0}\sum_{0\leq m <k}(f,h_{I^{(k)}})(f,h_{I^{(m)}})h_{I^{(m)}}(I^{(k)})h_{I^{(k)}}.
    \end{align*}
    We now use the latter representation to compute the Haar coefficients of $f^2$. First, note that $(f^2,h_I)=0$ whenever $I\neq I^{(n)}$ for any $n\in\mathbb{Z}$. Let us thus consider $I= I^{(n)}$ with $n<0$. In this case, applying~\eqref{E:counterex_low_reg_Haar_coeffs} and $h_{I^{(n)}}(I^{(k)})=|I^{(n)}|^{-1/2}=2^{\frac{n}{2}}$ we obtain
    \begin{align}
        (f^2,h_{I^{(n)}})&=\sum_{k>0}(f,h_{I^{(k)}})^2h_{I^{(n)}}(I^{(k)}) \notag \\
        &= \sum_{k=1}^\infty|I^{(k)}|^{2\alpha}h_{I^{(n)}}(I^{(k)}) 
        =2^{\frac{n}{2}}\sum_{k=1}^\infty 2^{-2\alpha k}=:2^{\frac{n}{2}}c_\alpha. \label{E:Haar_coeff_f2_lowreg_neg}
    \end{align}
    Note that $c_\alpha$ is indeed independent of $n$. In the case $I= I^{(n)}$ with $n\geq 0$, the Haar coefficient of $f^2$ becomes
    \begin{align}
        (f^2,h_{I^{(n)}})&=\sum_{k>n}(f,h_{I^{(k)}})^2h_{I^{(n)}}(I^{(k)}) +2\sum_{m=0}^{n-1}(f,h_{I^{(n)}})(f,h_{I^{(m)}})h_{I^{(m)}}(I^{(n)})\notag \\
        &=\sum_{k=n+1}^\infty 2^{-2\alpha k}2^{\frac{n}{2}}+2\sum_{m=0}^{n-1} 2^{-n\alpha}2^{-m\alpha}2^{\frac{m}{2}}\notag\\
        &=2^{\frac{n}{2}}\sum_{k=n+1}^\infty 2^{-2\alpha k}+22^{-n\alpha}\sum_{m=0}^{n-1} 2^{-m(\alpha-\frac{1}{2})}\notag\\
        &=2^{\frac{n}{2}}2^{-2n\alpha}\frac{1}{2^{2\alpha}}+22^{-n\alpha}2^{n(\frac{1}{2}-\alpha)}\frac{1-2^{-n(\frac{1}{2}-\alpha)}}{2^{\frac{1}{2}-\alpha}-1}\notag\\
        &=:2^{n(\frac{1}{2}-2\alpha)}(d_{\alpha}+e_{\alpha,n})
    \end{align}
    where $d_\alpha:=2^{-2\alpha}$ and
    \[
    e_{\alpha,n}:=2\frac{1-2^{-n(\frac{1}{2}-\alpha)}}{2^{\frac{1}{2}-\alpha}-1}\geq 2\frac{1-2^{-(\frac{1}{2}-\alpha)}}{2^{\frac{1}{2}-\alpha}-1}=:e_\alpha
    \]
    because $\alpha<s/2+1/4<1/2$.
    With the Haar coefficients in hand we can now give an explicit lower bound for the Sobolev seminorm
        \begin{align*}
            \| D^sf^2 \|_{L^2}^2 
            & =  \sum_{I\in \mathcal{D}} |I|^{-2s} |(f^2,h_I)|^2 =\sum_{m\in\mathbb{Z}}|I^{(m)}|^{-2s} |(f^2,h_{I^{(m)}})|^2 \\
            &\geq \sum_{m\geq 0}2^{2ms}e_{\alpha,n}^22^{m(1-4\alpha)}
            \geq e_\alpha^2 \sum_{m\geq 0}2^{m(2s-4\alpha+1)}.
        \end{align*}
        Since $\alpha<\frac{s}{2}+\frac{1}{4}$, the second series diverges and $f^2\notin H^s(\mathbb{R})$ as desired.
\end{proof}
\subsection{At criticality: \texorpdfstring{$s=1/2$}{[critical]}}
We learned in Proposition~\ref{P:algebra_high_reg} that the algebra property is satisfied when $s>1/2$, while the counterexample from Proposition~\ref{P:couterex_low_reg} required $s<1/2$. What happens when $s=1/2$? As it turns out, one needs to construct a slightly different function, whose Haar expansion uses again the right dyadic children $I^{(k)}$ that appeared in~\eqref{E:counterex_low_reg}.

\begin{proposition}\label{P:counterex_BMO}
    The space $H^{1/2}(\mathbb{R})$ is not an algebra. In particular, the function
    \begin{equation}\label{E:counterex_BMO}
    f:= \sum_{k>0} \frac{1}{2^{k/2}k^{\alpha/2}} h_{I^{(k)}},
    \end{equation}
    with $1<\alpha \leq 3/2$, belongs to $H^{1/2}(\mathbb{R})$ while $f^2\notin H^{1/2}(\mathbb{R})$.
\end{proposition}

\begin{proof}
    Note that, by definition of $f$, for any $I\in\mathcal{D}$
    \begin{equation}\label{E:Haar_coeff_BMO_fct}
        (f,h_I)=\begin{cases}
            \frac{1}{2^{k/2}k^{\alpha/2}} &\text{if }I=I^{(k)}\text{ for some }k> 0,\\
            0&\text{otherwise}.
        \end{cases}
    \end{equation}

    Thus,
    \begin{align*}
    \| D^{1/2}f \|_{L^2(\mathbb{R})}^2 
    &=  \sum_{I\in \mathcal{D}} |I|^{-1} |(f,h_I)|^2 
     = \sum_{k> 0} |I^{(k)}|^{-1} |(f,h_{I^{(k)}})|^2 \\
    & = \sum_{k> 0} 2^{k} \, \frac{1}{2^{k}k^{\alpha}}= \sum_{k> 0} \frac{1}{k^\alpha}
    \end{align*}
    which converges since $\alpha>1$. A similar computation gives
    \begin{equation}\label{E:counter_BMO_L2}
    \| f \|_{L^2(\mathbb{R})}^2 =  \sum_{I\in \mathcal{D}} |(f,h_I)|^2 = \sum_{k> 0} \frac{1}{2^{k}k^{\alpha}} < +\infty,
    \end{equation}
    hence $f\in H^{1/2}(\mathbb{R})$. 

    \medskip
    
    In order to verify that $f^2\notin H^{1/2}(\mathbb{R})$, we proceed as in Proposition~\ref{P:couterex_low_reg} and compute its Haar coefficients. Plugging the particular Haar coefficients of $f$ from~\eqref{E:counterex_BMO} in the Haar expansion~\eqref{E:Haar_exp_square} we obtain
    \begin{align*}
        f^2&=\sum_{m\in\mathbb{Z}}\sum_{k>\max\{0,m\}}(f,h_{I^{(k)}})^2h_{I^{(m)}}(I^{(k)})h_{I^{(m)}}
        +2\sum_{m> 0}\sum_{k>m}(f,h_{I^{(k)}})(f,h_{I^{(m)}})h_{I^{(m)}}(I^{(k)})h_{I^{(k)}}\\
        &=\sum_{m\in\mathbb{Z}}\sum_{k>\max\{0,m\}}(f,h_{I^{(k)}})^2h_{I^{(m)}}(I^{(k)})h_{I^{(m)}}
        +2\sum_{k> 0}\sum_{0< m <k}(f,h_{I^{(k)}})(f,h_{I^{(m)}})h_{I^{(m)}}(I^{(k)})h_{I^{(k)}}.
    \end{align*}
    In view of the latter expression, $(f^2,h_I)=0$ as long as $I\neq I^{(n)}$ for any $n\in\mathbb{Z}$. In the case $n\leq 0$, it follows from the above that
    \begin{align*}
        (f^2,h_{I{(n)}})&=\sum_{k>0}(f,h_{I^{(k)}})^2h_{I^{(n)}}(I^{(k)})=2^{\frac{n}{2}}\sum_{k>0}(f,h_{I^{(k)}})^2=2^{\frac{n}{2}}\|f\|_{L^2}^2.
    \end{align*}
    When $n>0$, 
    \begin{align*}
        (f^2,h_{I^{(n)}})& = \sum_{k=n+1}^\infty(f,h_{I^{(k)}})^2h_{I^{(n)}}(I^{(k)})+2\sum_{m=1}^n(f,h_{I^{(n)}})(f,h_{I^{(m)}})h_{I^{(m)}}(I^{(k)})\\
        &=2^{\frac{n}{2}}\sum_{k=n+1}^\infty2^{-k}k^{-\alpha}+2{\cdot}2^{-\frac{n}{2}}n^{-\frac{n}{2}}\sum_{m=1}^{n-1}2^{-\frac{m}{2}}m^{-\frac{\alpha}{2}}2^{\frac{m}{2}}\\
        &=2^{\frac{n}{2}}d_{\alpha,n}+2^{-\frac{n}{2}}n^{-\frac{n}{2}}e_{\alpha,n},
    \end{align*}    
    where 
    \[
      d_{\alpha,n}:=\sum_{k=n+1}^\infty2^{-k}k^{-\alpha}\quad\text{and}\quad e_{\alpha,n}:=\sum_{m=1}^{n-1}m^{-\frac{\alpha}{2}}.
    \]
    With the coefficients in hand, we find 
    \begin{align*}
        \| D^{1/2} f^2 \|_{L^2}^2 & = \sum_{n\in \mathbb{Z}} 2^n (f^2, h_{I^{(n)}})^2 
        \geq \sum_{n=1}^\infty 2^n (f^2, h_{I^{(n)}})^2 \\
        &\geq \sum_{n=1}^\infty 2^n (2^{\frac{n}{2}}d_{\alpha,n}+2^{-\frac{n}{2}}n^{-\frac{n}{2}}e_{\alpha,n})^2 \\
        & \geq  \sum_{n=1}^\infty \frac{1}{n^{\alpha}}e_{\alpha_n}^2.
    \end{align*}
    Finally, note that 
    \begin{align*}
        e_{\alpha,n} = \sum_{m=1}^{n-1} m^{-\frac{\alpha}{2}} 
         \geq  \int_2^{n-1} x^{-\frac{\alpha}{2}} \, dx = \frac{2}{\alpha-2}(n^{-\frac{\alpha}{2}+1}-2^{-\frac{\alpha}{2}+1})\simeq n^{-\frac{\alpha}{2}+1}.
    \end{align*}
    Since $\alpha \leq \frac{3}{2}$, then $2\alpha-2\leq 1$ and it follows that
    \[
    \| D^{1/2} f^2 \|_{L^2}^2\apprge \sum_{n=1}^\infty\frac{1}{n^{2\alpha-2}}= +\infty
    \]
    whence $f^2 \notin H^{1/2}(\mathbb{R})$.
\end{proof}
\subsection{High regularity: \texorpdfstring{$1/2<s<1$}{[large]}}\mbox{}
The embedding from Proposition~\ref{P:Morrey} is crucial to prove the validity of the algebra property in this ``high regularity range''. The key lemma in the dyadic proof relies on the following estimate.

\begin{lemma}\label{L:estimate_high_reg}
    For any fixed $I\in\mathcal{D}$, 
    \begin{equation}\label{E:estimate_high_reg}
        \sum_{J\subseteq I}|J|^{-2s}|(f^2,h_J)|^2\apprle_s |I|^{2s-1}\|f\|_{\dot{H}^s}^2\sum_{J\subseteq I}|J|^{-2s}|(f,h_J)|^2.
    \end{equation}
\end{lemma}

\begin{proof}
    By virtue of~\eqref{E:Haar_coef_square},
    \begin{equation}\label{E:high_reg_01}
      \sum_{J\subseteq I}|J|^{-2s}(f^2,h_J)^2\leq 2  \sum_{J\subseteq I}|J|^{-2s}\bigg(\sum_{\tilde{I}\subsetneq J}(f,h_{\tilde{I}})^2h_J(\tilde{I})\bigg)^2+4\sum_{J\subseteq I}|J|^{-2s}(f,h_J)^2\langle f\rangle_J^2.
    \end{equation}
    For the second term above, Proposition~\ref{P:Morrey} implies
    \begin{equation}\label{E:high_reg_avg_term}
        \sum_{J\subseteq I}|J|^{-2s}(f,h_J)^2\langle f\rangle_J^2\leq \sum_{J\subseteq I}|J|^{-2s}(f,h_J)^2\|f\|_{L^\infty}^2\apprle_s\|f\|^2_{\dot{H}^s(\mathbb{R})}\sum_{J\subseteq I}|J|^{-2s}(f,h_J)^2.
    \end{equation}
    To estimate the first term on the right-hand side of~\eqref{E:high_reg_01}, we write out the square and rearrange the summations to get
    \begin{align}
        &\sum_{J\subseteq I}|J|^{-2s}\bigg(\sum_{\tilde{I}\subsetneq J}(f,h_{\tilde{I}})^2h_J(\tilde{I})\bigg)^2
        =\sum_{J\subseteq I}|J|^{-2s}\sum_{\tilde{I}\subsetneq J}\sum_{\tilde{K}\subsetneq J}(f,h_{\tilde{I}})^2(f,h_{\tilde{K}})^2h_J(\tilde{I})h_J(\tilde{K}) \notag\\
        &=\sum_{\tilde{I}\subsetneq I}\sum_{\tilde{K}\subsetneq I}(f,h_{\tilde{I}})^2(f,h_{\tilde{K}})^2\sum_{\substack{\tilde{I}\subsetneq J\subseteq I\\ \tilde{K}\subsetneq J\subseteq I}}|J|^{-2s}h_J(\tilde{I})h_J(\tilde{K})\notag\\ 
        &=\sum_{\tilde{I}\subsetneq I}|\tilde{I}|^{-2s}(f,h_{\tilde{I}})^2\sum_{\tilde{K}\subsetneq I}|\tilde{K}|^{-2s}(f,h_{\tilde{K}})^2\sum_{\substack{\tilde{I}\subsetneq J\subseteq I\\ \tilde{K}\subsetneq J\subseteq I}}|J|^{-2s}h_J(\tilde{I})h_J(\tilde{K})|\tilde{I}|^{2s}|\tilde{K}|^{2s}.\label{E:high_reg_sq_term}
    \end{align}
    We now focus on estimating the last summation. For this, we first use the fact that $h_J(\tilde{I})\leq |J|^{-1/2}$ because $\tilde{I}\subsetneq J$, and the same for $\tilde{K}$. Afterwards, we split the sum in two terms depending on whether $|\tilde{I}|<|\tilde{K}|$ or $|\tilde{I}|\geq|\tilde{K}|$. In this way we obtain
    \begin{align}
        \sum_{\substack{\tilde{I}\subsetneq J\subseteq I\\ \tilde{K}\subsetneq J\subseteq I}}|J|^{-2s}h_J(\tilde{I})h_J(\tilde{K})|\tilde{I}|^{2s}|\tilde{K}|^{2s}
        &\leq \sum_{\substack{\tilde{I}\subsetneq J\subseteq I\\ \tilde{K}\subsetneq J\subseteq I}}|J|^{-2s-1}|\tilde{I}|^{2s}|\tilde{K}|^{2s}\notag\\
        &\leq \sum_{\substack{\tilde{I}\subsetneq J\subseteq I\\ \tilde{K}\subsetneq J\subseteq I\\ |\tilde{I}|<|\tilde{K}|}}|J|^{-2s-1}|\tilde{I}|^{2s}|\tilde{K}|^{2s}+\sum_{\substack{\tilde{I}\subsetneq J\subseteq I\\ \tilde{K}\subsetneq J\subseteq I\\ |\tilde{I}|\geq |\tilde{K}|}}|J|^{-2s-1}|\tilde{I}|^{2s}|\tilde{K}|^{2s}\notag\\
        &\leq \sum_{\tilde{K}\subsetneq J\subseteq I}|J|^{-2s-1}|\tilde{K}|^{4s}+\sum_{\tilde{I}\subsetneq J\subseteq I}|J|^{-2s-1}|\tilde{I}|^{4s}.\label{E:high_reg_02}
    \end{align}
    To show that the terms above are bounded, note that 
    \begin{align*}
       \sum_{\tilde{K}\subsetneq J\subseteq I}|J|^{-2s-1}|\tilde{K}|^{4s}
        &=\sum_{\tilde{K}\subsetneq J\subseteq I}\bigg(\frac{|\tilde{K}|}{|J|}\bigg)^{2s}\frac{|\tilde{K}|^{2s}}{|J|}\\
        &\leq \sum_{\tilde{K}\subsetneq J\subseteq I}\bigg(\frac{|\tilde{K}|}{|J|}\bigg)^{2s}|J|^{2s-1}
        \leq|I|^{2s-1}\sum_{\tilde{K}\subsetneq J\subseteq I}\bigg(\frac{|\tilde{K}|}{|J|}\bigg)^{2s}.
    \end{align*}
    Further, we rewrite the index summation set as follows: For fixed $\tilde{K}$ and $I$ there are $\tilde{k}>k$ such that $\ell(\tilde{K})=2^{-\tilde{k}}$ and $\ell(I)=2^{-k}$. Then,
    \begin{equation}\label{E:high_reg_index_set}
    \{J\in\mathcal{D}\colon \tilde{K}\subsetneq J\subset I\}=\{J\colon \ell(J)=2^{-m}, k<m\leq\tilde{k}\}
    \end{equation}
    so that
    \begin{equation*}
        \sum_{\tilde{K}\subsetneq J\subseteq I}\bigg(\frac{|\tilde{K}|}{|J|}\bigg)^{2s}
        =\sum_{m=k}^{\tilde{k}}2^{2s(m-\tilde{k})}
        =\sum_{j=0}^{\tilde{k}-k}2^{-2sj}<\sum_{j=0}^{\infty}2^{-2sj}
    \end{equation*}
    which is uniformly bounded. An analog computation with $\tilde{I}$ instead of $\tilde{K}$ shows the same bound for the second term in~\eqref{E:high_reg_02}. Plugging back into~\eqref{E:high_reg_sq_term} yields
    \begin{align*}
        \sum_{J\subseteq I}|J|^{-2s}\bigg(\sum_{\tilde{I}\subsetneq J}(f,h_{\tilde{I}})^2h_J(\tilde{I})\bigg)^2
        &\apprle |I|^{2s-1}\sum_{\tilde{I}\subsetneq I}|\tilde{I}|^{-2s}(f,h_{\tilde{I}})^2\sum_{\tilde{K}\subsetneq I}|\tilde{K}|^{-2s}(f,h_{\tilde{K}})^2\\
        &\leq |I|^{2s-1}\|f\|_{\dot{H}^s(\mathbb{R})}^2\sum_{\tilde{I}\subsetneq I}|\tilde{I}|^{-2s}(f,h_{\tilde{I}})^2.
    \end{align*}
    Combined with~\eqref{E:high_reg_avg_term} and~\eqref{E:high_reg_01} one finally arrives at~\eqref{E:estimate_high_reg}.
\end{proof}

We are now in a position to conclude the algebra property of $H^s(\mathbb{R})$.

\begin{proposition}\label{P:algebra_high_reg}
    For any $1/2<s<1$, the space $H^s(\mathbb{R})$ is an algebra, that is, $f\in H^s(\mathbb{R})$ implies $f^2\in H^s(\mathbb{R})$.
\end{proposition}
\begin{proof}
    Let $\mathcal{D}_1$ be a partition of $\mathbb{R}$ by dyadic intervals of length $1/2$. 
    By virtue of Lemma~\ref{L:norm_equivalence}, Lemma~\ref{L:estimate_high_reg} and Proposition~\ref{P:Morrey} we obtain
    \begin{align*}
        \|f^2\|_{H^s(\mathbb{R})}^2&\leq 2\|f^2\|_{L^2(\mathbb{R})}^2+2\sum_{\substack{I\in\mathcal{D}\\|I|<1}}|I|^{-2s}(f^2,h_I)|^2\\
        &=2\|f^2\|_{L^2(\mathbb{R})}^2+2\sum_{I\in\mathcal{D}_1}\sum_{J\subseteq I}|J|^{-2s}|(f^2,h_J)|^2\\
        &\apprle_s \|f\|_{L^\infty(\mathbb{R})}^2\|f\|_{L^2(\mathbb{R})}^2+\|f\|_{H^s(\mathbb{R})}^2\sum_{I\in\mathcal{D}_1}\sum_{J\subseteq I}|J|^{-2s}|(f,h_J)|^2\\
        &=\|f\|_{L^\infty(\mathbb{R})}^2\|f\|_{L^2(\mathbb{R})}^2+\|f\|_{H^s(\mathbb{R})}^2\sum_{\substack{I\in\mathcal{D}\\|I|<1}}|J|^{-2s}|(f,h_J)|^2\\
        &\apprle_s \|f\|_{H^s(\mathbb{R})}^2\Big(\|f\|_{L^2(\mathbb{R})}^2+\sum_{\substack{I\in\mathcal{D}\\|I|<1}}|J|^{-2s}|(f,h_J)|^2\Big)
        \leq \|f\|_{H^s(\mathbb{R})}^4.
    \end{align*}
\end{proof}
\bibliographystyle{amsplain}
\bibliography{Dyadic_Sobolev_Refs}
\end{document}